\documentclass[letterpaper,12pt,reqno]{amsart}
\usepackage{graphicx,amsmath,amssymb, amsthm}

\usepackage{srcltx}

\usepackage{fullpage}

\theoremstyle{plain}
\newtheorem{X}{X}[section]
\newtheorem{corollary}[X]{Corollary}

\newtheorem{lemma}[X]{Lemma}
\newtheorem{proposition}[X]{Proposition}
\newtheorem{theorem}[X]{Theorem}

\newtheorem{hypothesis}[X]{Hypothesis}
\theoremstyle{remark}
\newtheorem{remark}[X]{Remark}

\newcommand{\A}{\mathcal A}

\newcommand{\ZZ}{\mathbb{Z}}

\renewcommand{\r}{\mathbf{r}}
\renewcommand{\d}{{\rm d}}
 \newcommand{\e}{{\rm e}}
\newcommand\mycom[2]{\genfrac{}{}{0pt}{}{#1}{#2}}

\title{}
\author{}
\address{}
\email{}
\address{}
\email{}
\date{\today}

\begin{document}
\title[
]{Major arcs and moments of arithmetical sequences}

\date{\today}
 
\author{R.\ de la Bret\`eche}
\address{
Institut de Math\'ematiques de Jussieu-Paris Rive Gauche\\
Universit\'e  Paris Diderot\\ \newline \rule[0ex]{0ex}{0ex}\hspace{8pt}
Sorbonne Paris Cit\'e, UMR 7586\\
Case Postale 7012\\
F-75251 Paris CEDEX 13\\ France}
\email{regis.delabreteche@imj-prg.fr}

\author{D. Fiorilli}
\address{D\'epartement de math\'ematiques et de statistique, Universit\'e d'Ottawa, \newline
\rule[0ex]{0ex}{0ex}\hspace{8pt} 585 King Edward, Ottawa, Ontario, K1N 6N5, Canada}
\email{daniel.fiorilli@uottawa.ca}
 \dedicatory{\`A la m\'emoire de notre ami Kevin Henriot}

\maketitle

\begin{abstract}

We give estimates for the first two moments of arithmetical sequences in progressions.  Instead of using the standard approximation, we work with a generalization of Vaughan's major arcs approximation which is similar to that appearing in earlier work of Browning and Heath-Brown on norm forms. 
We apply our results to the sequence $\tau_k(n)$, and obtain unconditional results in a wide range of moduli.


\end{abstract}

\section{Introduction}

Given a sequence $\A = \{ f(n) \}_{n\geq 1}$ of real numbers, a natural measure of its irregularity in progressions modulo $q$ is the probabilistic variance
\begin{equation}
\sum_{1\leq a \leq q} \Big(\sum_{\substack{n\leq x \\ n\equiv a \bmod q}} f(n)-\frac 1{\phi(q/(q,a))}\sum_{\substack{n\leq x \\ (n,q)=(a,q) }} f(n)\Big)^2. 
\label{eqdefvarf}
\end{equation} 
This quantity has been widely studied in the literature for many interesting arithmetical sequences (see for example \cite{Mo,Ho,H75,V98,V05,N15,KR16,HS16}). Notably, Rodgers and Soundararajan \cite{RS16} have recently established an asymptotic for the $k$-th divisor function $f(n)=\tau_k(n)$, on average over moduli $q$ in certain ranges. Their results confirm predictions coming from function fields \cite{KRRR15}, and are related to the estimation of moments of Dirichlet $L$-functions.

The quantity \eqref{eqdefvarf} is the probabilistic variance of $f(n)$, in the sense that the approximation
$$ \frac 1{\phi(q/(q,a))}\sum_{\substack{n\leq x \\ (n,q)=(a,q) }} f(n) $$
is the quantity depending only on $(a,q)$ that minimizes it. One can ask whether there exists a natural approximation that depends more strongly on $a$
which results in a smaller variance. In the case of the sequence of primes, that is $f(n) = \Lambda(n)$, this was shown possible by Vaughan \cite{V03a,V03b}, by considering the contribution of the major arcs in an application of the circle method. 

In the current paper we work with a generalization of Vaughan's approximation to general arithmetical sequences $\A = \{ f(n) \}_{n\geq 1}$. Such an approximation was used in the work of Browning and Heath-Brown \cite[Section 4]{BHB12}, and was crucial in order to obtain an asymptotic expression for the bilinear forms of trace functions they considered.
The resulting major arcs approximations $\rho_R(x;q,a),$ $\rho^*_R(x;q,a)$, which in our context are built under two natural sieve hypotheses and depend on a parameter $R$,
 are defined below in \eqref{eqdefrho}. Note that these approximations depend only on the distribution of the sequence $f(n)$ in arithmetic progressions, rather than depending intrinsically on $f(n)$. By this we mean that $\rho_R(x;q,a)$ depends only on the quantities $h_j(d)$ appearing in Hypothesis \ref{hyphd} below, which measure the size of $\A$ in the arithmetic progression $0\bmod d$. Before we give our general results let us consider the case of the sequence $f(n)=\tau_k(n)$, the $k$-th divisor function, with $k\geq 2$. We obtain results for the first two moments. 

\begin{theorem}[First moment of $\tau_k$] There exist absolute constants $\eta,\eta'>0$ such that for every $k\geq 2$ and in the range $-x/N < a \leq x $, $N+1\leq \min\{x^{1/2-\eta/k}, R\}$, we have the bound
$$\sum_{x/N<q\leq x} \Big(\sum_{\substack{a<n\leq x \\ n\equiv a \bmod q}} \tau_k(n)-\rho^*_R(x;q,a)\Big)\ll  NR(\log R \log x)^{k}  +x^{1-\eta'/k}\tau_2(|a|).$$

\label{thmomenttauk1} 
\end{theorem}

\begin{theorem}[Variance of $\tau_k$]\label{thmoment2tauk} 
There exists an absolute constant $\eta>0$ such that for every $k\geq 2$ and in the range $N \leq R\leq x^{\eta/k}$ 
we have
\begin{multline}\label{estDeltatauk}\frac {2N}x\sum_{\frac x{2N}<q\leq  \frac xN} \sum_{1\leq a \leq q} \Big(\sum_{\substack{n\leq x \\ n\equiv a \bmod q}} \tau_k(n)-\rho_R(x;q,a)\Big)^2 =   \\
C_k x \frac{(\log x)^{k^2-1}}{(k^2-1)!} \Big(1-\Big(\frac{\log R}{\log x} \Big)^{(k-1)^2}P_k\Big(\frac{\log R}{\log x} \Big)+ \frac 1{\log x}Q_k\Big( \frac 1{\log x}\Big)\Big) + O\big( x(\log x)^{2k-2}(\log R)^{k^2-2k}\big), 
\end{multline}
where $Q_k$ and $P_k$ are polynomials of degree $k^2-2$ and $2k-2$ respectively, which are defined in Section \ref{section variance tauk}. Also
$$ C_k := \prod_p \Big( 1-\frac 1p \Big)^{k^2}\Big( \sum_{\nu\geq 0 }\frac{  \tau_{k}(p^{\nu})^2}{ p^{\nu}}\Big) . $$
\end{theorem}

\begin{remark}
Based on their results on the function field counterpart, Keating, Rodgers, Roddity-Gershon and Rudnick \cite{KRRR15} have made the following conjecture on the probabilistic variance:
$$\sum_{\substack{1\leq a \leq q \\ (a,q)=1}} \Big(\sum_{\substack{n\leq x \\ n\equiv a \bmod q}} \tau_k(n)-\frac 1{\phi(q)}\sum_{\substack{n\leq x \\ (n,q)=1}} \tau_k(n)\Big)^2 \sim  a_k(q)\gamma_k(c) x (\log q)^{k^2-1},  $$
where $x=q^{c+o(1)}$,
$$  a_k(q):= \prod_{p\mid q}\Big( 1-\frac 1p \Big)^{k^2}\prod_{p\nmid q} \Big( 1-\frac 1p \Big)^{k^2}\Big( \sum_{\nu\geq 0 }\frac{  \tau_{k}(p^{\nu})^2}{ p^{\nu}}\Big),$$
and 
$$ \gamma_k(c)= \frac 1{k!G(k+1)^2} \int_{[0,1]^k} \delta(w_1+...+w_k-c) \Delta(w)^2 d^kw. $$
Here $\delta$ is the Dirac delta-function, $\Delta(w):= \prod_{i<j}(w_i-w_j)$ is the Vandermonde determinant, and $G(k+1)=(k-1)!(k-2)!\cdots 1!$. On average over $q\asymp Q$ and for  $\tfrac{\log x}{\log Q} \rightarrow c\in [\delta,1+\tfrac{2}k -\delta]$, their conjecture was confirmed by the work of Rodgers and Soundararajan \cite{RS16}. In the current paper we are interested in the range $c\in [1,1+O(\tfrac 1k))$, in which \cite[Lemma 6]{KRRR15} implies that
$$\gamma_k(c)= \frac{c^{k^2-1}}{(k^2-1)!} +R_k(c),$$
where $R_k(c)$ is the polynomial defined by
$$R_k(c):= {\rm Res}_{z=0}{\rm Res}_{\substack{s_1=0 \\ s_2=0}} e^{z} e^{c(s_1+s_2-z)} \frac{(s_1-z)^k(s_2-z)^k}{z^{k^2} s_1^ks_2^k (s_1+s_2-z)^{k^2}}.$$
We will see in Remark \ref{remark following 6.1} that the polynomial $P_k$ in Theorem \ref{thmoment2tauk} is exactly of this form, and thus our results are very similar to their function field counterparts\footnote{The slight difference in the constant $a_k(q)$ versus our constant $C_k$ comes from the fact that we are considering all residue classes $a\bmod q$, rather than requiring $(a,q)=1$.}. In particular if we take $R=N=x/Q$ in Theorem 1.2 then we recover the expression obtained 
in \cite[Theorem 1]{RS16}. Indeed, setting $v:=\log R/\log x = 1-1/c$, we have that \eqref{eqcalculgammak} implies the formula
\begin{align*}(\log x)^{k^2-1}\frac{1-v^{(k-1)^2}P_k(v)}{(k^2-1)!}&
=\gamma_k(c)(\log q)^{k^2-1}.
\end{align*}
As for larger values $R>N$, we will see that by monotonicity (see Remark \ref{remark after variance general} (3) below), they will result in a smaller variance. 


%

\end{remark}

We now describe our general results. We will consider arithmetical sequences $\mathcal A=\{f(n)\}_{n\geq 1}$ satisfying two hypotheses. The first is analogous to a standard sieve hypothesis, and the second describes the distribution of $\mathcal A$ in arithmetic progressions. These hypotheses include  two real functions $Q(x)$ and $L(x)$; for now one can think of $Q(x)$ as a small power of $x$ and of $L(x)$ as a power of $(\log x)$.

\begin{hypothesis}[Arithmetical Sequence]
\label{hyphd}
There exists an integer $J\geq 0$, arithmetical functions\footnote{The $h_j$ need not be multiplicative.} $h_j$ and monotonic smooth functions $u_j:\mathbb R_{\geq 0} \rightarrow \mathbb R $ with $0\leq j< J$ such that uniformly for $1\leq d\leq x$ we have
$$ \A_d(x):= \sum_{\substack{n\leq x \\ d\mid n}} f(n)  =\sum_{0\leq j< J} \frac{h_j(d)}d U_j(x) + O\left( \frac{U_0(x)}{L(x)(\log Q(x))^2}\right), $$
where $U_j(x):=\int_0^x u_j.$
\end{hypothesis}

For ease of notation, we set $u_j(t)=0$ for $t<0$.
\begin{hypothesis}[Equidistribution in arithmetic progressions]
\label{hypequir}
For $x\geq 1$ we have the bound
$$ \sum_{q\leq Q(x)} \max_{y\leq x} \max_{1\leq a\leq q} \Big| \sum_{\substack{n\leq y \\ n\equiv a \bmod q }}f(n)- \frac 1{\phi(q/(q,a))} \sum_{\substack{ n\leq y \\ (n,q)=(a,q)}} f(n) \Big| \ll \frac {U_0(x)}{L(x)}.$$

\end{hypothesis}

 Under these hypotheses, we define the major arcs contributions
$$F_R(n):=\sum_{0\leq j< J}u_j(n)\sum_{ r\leq R }\frac{(h_j*\mu)(r)  }{
\phi(r/(r,n))}\mu\Big(\frac{r}{(r,n)}\Big),$$
and the approximations
\begin{equation}
\rho_R(x;q,a):= \sum_{\substack{ n\leq x \\ n\equiv a \bmod q}}F_R(n); \hspace{1cm}\rho^*_R(x;q,a):= \sum_{\substack{ a<n\leq x \\ n\equiv a \bmod q}}F_R(n).
\label{eqdefrho}
\end{equation}

We can now give an estimate for the first two moments of the difference 
$ \Delta_R(n):=f(n)-F_R(n)$ in arithmetic progressions.

\begin{theorem}\label{thmoment1} Assume Hypotheses \ref{hyphd} and \ref{hypequir}.
In the ranges $-x/N < a \leq x $, $N+1\leq \min\{Q(x), R\}$,   we have
$$\sum_{x/N<q\leq x} \Big(\sum_{\substack{a<n\leq x \\ n\equiv a \bmod q }}f(n)-\rho^*_R(x;q,a)\Big) 
\ll N(\log R)u_1^*(x,R) +\tau(a)\frac{U_0(x)}{L(x)},$$
where
\begin{equation}\label{defu1*xR} u_1^*(x,R):=\sum_{0\leq j< J} (|u_j(x)|+1) \sum_{\substack{ r\leq R}}\frac{|(h_j*\mu)(r) |r }{
\phi(r )}.
\end{equation}
The same bound holds for the sum over the moduli $q\leq N$. Moreover, for $a\geq 0$ the second condition can be relaxed to $N\leq \min\{Q(x), R\}$.
\end{theorem}

\begin{remark} 
\
\begin{enumerate}
\item
We recover \cite[Theorem 1.5]{F15} in the case $f =\Lambda$, $h_0=\delta$, $J=1$, $u_0=1$, for which Hypothesis \ref{hyphd} is the prime number theorem and Hypothesis \ref{hypequir} follows from the Bombieri--Vinogradov Theorem, with $L(x)=(\log x)^A$ and $Q(x)=\sqrt{x}/(\log x)^B$. 
 
\item
As in \cite{F15}, we can hope to obtain an asymptotic estimate for the sum over moduli $q\leq x/N$. This would require additional hypotheses on $f$ in order to estimate certain sums of arithmetical functions. We preferred not to pursue this further.

\item Comparing with \cite[Theorem 4.1*]{F13}, we see that the first moment in Theorem \ref{thmoment1} is of much smaller size. The reason for this improved bound is that the approximation $\rho^*(x;q,a)$ captures the discrepancies of the distribution of $f(n)$ in arithmetic progressions. 

\end{enumerate}
\end{remark}

For an arithmetical function $f$, define the $2$-norm
$$\|f\|_2^2:=\sum_{\substack{ n\leq x }}f(n)^2.$$ 
We will see that the variance is dominated by the $2$-norm of $ \Delta_R(n)$, and we will give an estimate for this quantity.

\begin{theorem}\label{thmoment2} 
Assume Hypotheses \ref{hyphd} and \ref{hypequir} and recall the notation \eqref{defu1*xR}. If \linebreak$ N\leq \min\{Q(x), R,U_0(x)/(L(x)u_1^*(x,R) (\log R) ) \} $, then we have
\begin{equation}\label{estVDelta}\sum_{\frac x{2N}<q\leq \frac xN} \sum_{1\leq a\leq q} \Big(\sum_{\substack{a<n\leq x \\ n\equiv a \bmod q }}f(n)-\rho^*_R(x;q,a)\Big)^2 = \frac x{2N} \|\Delta_R\|_2^2 \Bigg\{ 1  +O\left(\frac Nx+\frac{NU_0(x)(\log x)^{\frac 32} }{x^{\frac 12}\|\Delta_R\|_2 L(x)} \right)\Bigg\}.\end{equation}
Moreover, for $R\geq 1$ we have
\begin{equation} \begin{split}
\|\Delta_R\|_2^2 &= \|f\|_2^2 - \sum_{\substack{ r\leq R  }}  \frac{1}{\phi(r)} \int_{0}^{x}
\Big(\sum_{0\leq j  < J}(h_j*\mu)(r) u_j(t)\Big)^2\d t\\&\quad +O \left(u_2^*(x,R)^2+ \frac{U_0(x)u_2^*(x,R)}{L(x)(\log Q(x))^2}\right), \end{split}
\label{estM2DELTA}
\end{equation}
where 
$$ u_2^*(x,R)  :=\sum_{0\leq j  < J}(|u_j(x)|+1)\sum_{\substack{ r\leq R}}\frac{|(h_j*\mu)(r) |r^2}{
\phi(r )^2} 
.$$
\end{theorem}

\goodbreak
\begin{remark}
\
\label{remark after variance general}
\begin{enumerate}

\item We recover \cite[Theorems 3,4]{V03a} by taking $f=\Lambda$.

\item We recover \cite[Theorem 2]{H75} by taking $f=\mu$ which corresponds for us to a very particular case since $u_1^*(x,R)=u_2^*(x,R)=0$. Indeed, we can make the choices $J=1$, $h_0=0$, $u_0=1$, $L(x)=(\log x)^{B}$, $Q(x)=\sqrt{x}/(\log x)^{A}$, and hence $F_R= 0$.

\item The estimate \eqref{estM2DELTA} is a refinement of \cite[Lemma 6]{BHB12}, which is based on the large sieve.

\item The second term on the right hand side of \eqref{estM2DELTA} is negative and monotonic in $R$. In other words, the larger $R$ is, the more precise the approximation $F_R(n)$ becomes, and the smaller the variance is.

\item  If $J=1$, then the main terms on the right hand side of \eqref{estM2DELTA} become
\begin{equation}
 \sum_{n\leq x} f(n)^2 -  \sum_{\substack{ r\leq R  }}  \frac{(h_0*\mu)(r)^2}{\phi(r)} \int_{0}^{x}  u_0(t)^2\d t. 
 \label{equation variance simple}
\end{equation}
 The ratio between $\|f\|_2^2$ and $\int_{0}^{x}  u_0(t)^2\d t$ is a measure of the irregularity of $f$, and is compensated by the sum over $r$ in such a way that \eqref{equation variance simple} stays positive. 
\end{enumerate}
\end{remark}

\section{General major arcs approximation}
\label{section Vaughan}

The goal of this section is to justify the definition of the major arcs approximations $\rho_R(x;q,a), \rho^*_R(x;q,a)$ given in \eqref{eqdefrho}. Our argument results in an approximation similar to that of Browning and Heath-Brown \cite[Section 4]{BHB12}, and generalizes Vaughan's approximation for $\Lambda(n)$ \cite{V03a,V03b} to general arithmetic sequences $f(n)$ satisfying Hypotheses \ref{hyphd} and \ref{hypequir}. We will combine the circle method with an argument of La Bret\`eche and Granville \cite{BG14}, which determine the resonance of $f(n)$ with the harmonics $e(bn/r)$. 
We define the counting functions
$$  \A (x;q,a):=\sum_{\substack{n\leq x\\n\equiv a\bmod q}}f(n),\qquad \A^* (x;q,a):=\sum_{\substack{a<n\leq x\\n\equiv a\bmod q}}f(n).$$
Recall also that 
$$ \A_d(x)= \sum_{\substack{n\leq x \\ d\mid n}} f(n). $$
We first see that from Hypotheses \ref{hyphd} and \ref{hypequir} we can deduce an asymptotic for  
$\A(x;q,a)$ in terms of the function
\begin{equation}
g_j(q,a):= \frac 1{\phi(q/(q,a))} \sum_{d\mid q/(q,a)} \mu(d) \frac{h_j((q,a)d)}{(q,a)d}. 
\label{eqdefgaq}
\end{equation} 
The argument is due to Granville and Soundararajan \cite{GS}.

\begin{corollary}
Under Hypotheses \ref{hyphd} and \ref{hypequir}, we have uniformly for $a\in \ZZ$ and $x\geq 1$ that
$$ \sum_{q\leq Q(x)} \Bigg| \sum_{\substack{n\leq x \\ n\equiv a \bmod q }}f(n)- \sum_{0\leq j< J}g_j(q,a) U_j(x) \Bigg| \ll \tau(a) \frac {U_0(x) }{L(x)}.$$
\label{corBV}
\end{corollary}

\begin{proof}
We have that
$$ 
 \frac 1{\phi(q/(q,a))} \sum_{\substack{ n\leq x \\ (n,q)=(a,q)}} f(n) =  \frac 1{\phi(q/(q,a))} \sum_{\substack{ d\mid q/(q,a)}} \mu(d) \A_{(q,a)d}(x),
$$
and hence
\begin{multline*} \sum_{q\leq Q(x)} \Big| \sum_{\substack{n\leq x \\ n\equiv a \bmod q }}f(n)- \sum_{0\leq j< J}g_j(q,a) U_j(x) \Big|
\\  \leq 
\sum_{q\leq Q(x)} \Big| \sum_{\substack{n\leq x \\ n\equiv a \bmod q }}f(n)- \frac 1{\phi(q/(q,a))} \sum_{\substack{ n\leq x \\ (n,q)=(a,q)}} f(n)  \Big| \\ + \sum_{q\leq Q(x)} \Big| \frac 1{\phi(q/(q,a))} \sum_{\substack{ d\mid q/(q,a)}} \mu(d) \A_{(q,a)d}(x)-  \sum_{0\leq j< J}g_j(q,a) U_j(x) \Big|.
\end{multline*}
The first term is $\ll U_0(x)/L(x)$ by Hypothesis \ref{hypequir}, and by Hypothesis \ref{hyphd} the second is
\begin{align*}&=\sum_{q\leq Q(x)} \Big|\frac 1{\phi(q/(q,a))}    \sum_{\substack{ d\mid q/(q,a)}} \mu(d) \Big(\A_{(q,a)d}(x)- \sum_{0\leq j< J}  \frac{h_j((q,a)d)}{(q,a)d} U_j(x) \Big)\Big| \cr& \quad\ll \frac{U_0(x)}{L(x)(\log Q(x))^2} \sum_{q\leq Q(x)} \frac{\tau(q/(q,a))}{\phi(q/(q,a))}.
\end{align*}
The last sum over $q$ is $$\leq \sum_{m\mid a}\sum_{\substack{ q\leq Q(x)\\ m\mid q}} \frac{\tau(q/m)}{\phi(q/m)}\leq \sum_{m\mid a}\sum_{\substack{ q\leq Q(x)/m}} \frac{\tau(q )}{\phi(q )}\ll \tau(a) (\log Q(x))^2 .$$
\end{proof}

Coming back to the circle method, we use the notations of \cite{BG14} as follows:
$$E_f\Big(x;\frac{b}{r}\Big):= \sum_{n\leq x} e\Big(\frac{bn}{r}\Big) f(n);$$
\begin{equation}\label{defRf}\begin{split}
  R_f(x;a/q)&:=\sum_{1\leq b\leq q}\e\Big(\frac{ab}{
q}\Big)\Big\{\sum_{\substack{n\leq x\\ n\equiv b \bmod q }}f(n)-\frac{1}{ \phi(q/(q,b))}
\sum_{\substack{n\leq x\\(n,q)=(b,q)}}f(n)
\Big\}
\cr&= \sum_{m\mid q}\sum_{\substack{c=1\\(c,m)=1}}^{m}\e\Big(\frac{ac}{
m}\Big)
\Big(\sum_{\substack{n\leq xm/q\\ n\equiv c \bmod m }}
f\Big(\frac{nq}{ m}\Big)-\frac{1}{ \phi(m)}\sum_{\substack{n\leq xm/q\\ (n,m)=1}}
f\Big(\frac{nq}{ m}\Big)  \Big).\end{split} \end{equation}

If $f$ is well-distributed in arithmetic progressions of small moduli, then $R_f(x;a/q)$ will be an error term. More precisely, Hypothesis \ref{hypequir} implies that
\begin{equation}\qquad\qquad
 \sum_{q\leq Q } |R_f(x;a/q)| \ll \frac{Q U_0(x)}{L(x)} \qquad\qquad (Q\leq  Q(x)), 
 \label{eqmajRfsousH1}
\end{equation}
and hence if we have equidistribution in small arithmetic progressions (that is $Q$ is negligible compared to $L(x)$), then $R_f(x;a/q)$ is small on average. In other words, a Siegel-Walfisz hypothesis implies that $R_f(x;a/q)$ is negligible. 

We reproduce \cite[Lemme 3.2]{BG14} and its proof.

\begin{lemma}[\cite{BG14}]\label{lemGene} For $(b,r)=1$ and $x\geq 1$, we have the identity
\begin{align*}E_f\Big(x;\frac{b}{r}\Big)&=\sum_{m\mid r}\frac{\mu(r/m) m }{
\phi(r)}\sum_{\substack{n\leq x\\ m\mid n}}
f(n) +R_f\Big(x;\frac{b}{r}\Big).\end{align*}
\end{lemma}


%

\begin{proof}
We have
\begin{align*}E_f\Big(x;\frac{b}{r}\Big)
&=\sum_{c=1}^r\sum_{\substack{n\leq x\\ n\equiv c\bmod r }}
f(n)\e\Big(\frac{bn}{r}\Big)\cr& =\sum_{m\mid r}\sum_{\substack{a=1\\(a,m)=1}}^{m}\e\Big(\frac{ab}{ m}\Big)
\sum_{\substack{n\leq x\\ r/m\mid n\\ nm/r\equiv a \bmod m }}
f(n)
\cr& =\sum_{m\mid r}\sum_{\substack{a=1\\(a,m)=1}}^{m}\e\Big(\frac{ab}{ m}\Big)
\sum_{\substack{n\leq xm/r\\ n\equiv a\bmod m }}
f\Big(\frac{nr}{ m}\Big)
.\end{align*}
Therefore we can write
$$E_f\Big(x;\frac{b}{ r}\Big)=M_f\Big(x;\frac{b}{r}\Big)+R_f\Big(x;\frac{b}{r}\Big)$$
where $R_f(x; {b}/{r} ) $ is defined in \eqref{defRf} and
\begin{align*}M_f\Big(x;\frac{b}{r}\Big)&:= \sum_{m\mid r}\frac{\mu(m)}{
\phi(m)}\sum_{\substack{n\leq xm/r\\ (n,m)=1}} f\Big(\frac{nr}{m}\Big).   \end{align*} 
M\"obius inversion gives the desired result.
\end{proof}

By Lemma \ref{lemGene}, if $R_f(x,b/r)$ is negligible and $\beta$ is small, then $E_f(x;b/r+\beta)$ is well approximated by the quantity
$$
\sum_{k\mid r}\frac{\mu(r/k) k }{
\phi(r )}\sum_{\substack{n\leq x\\ k\mid n}}
f(n) \e (n\beta).
$$
Indeed, summation by parts gives the identity
\begin{align*}
E_f(x;b/r+\beta)&=\e(\beta x)  
E_f(x;b/r )-2\pi i\beta \int_1^x\e(\beta t) E_f(t;b/r )\d t 
\cr&=\sum_{k\mid r}\frac{\mu(r/k) k }{
\phi(r )}\sum_{\substack{n\leq x\\ k\mid n}}
f(n) \e (n\beta) +O\Big(\big( 1+|\beta| x\big)R_f^*\Big(x;\frac{b}{ r}\Big) \Big)
\end{align*}
where $$R_f^*\Big(x;\frac{b}{ r}\Big):=\max_{t\leq x}
\Big| R_f\Big(t;\frac{b}{ r}\Big)\Big|.$$
By \eqref{eqmajRfsousH1}, Hypothesis \ref{hypequir} gives the following bound on this error term, under the condition $r\leq Q(x)$:  $$R^*_f\Big(x;\frac{b}{ r}\Big)\ll r \frac {U_0(x)}{L(x)}.$$
This gives an idea of the admissible size of $\beta$ relative to $R_f$. For example, we can take $|\beta|\leq 1/rR$ with $r\leq \min\{ Q(x),\sqrt{L(x)}\}$ and $R=x/\sqrt{L(x)}$, and obtain an error term $\ll U_0(x)/\sqrt{L(x)}.$

We now apply the circle method to $f(m)$, and extract the major arcs contribution. Writing
$$f(m)=\int_0^1 \e(-m\theta) \Big(\sum_{n\leq x} f(n)\e(n\theta)\Big)\d \theta,$$
we can hope (at least heuristically) that a large contribution comes from the arcs $]b/r-1/S, b/r-1/S[$ where $r\leq R$ and $(b,r)=1$.

The sum of all major arcs contributions to $f(m)$ is therefore given by
\begin{equation}\label{f(m)1}\begin{split}
& \sum_{\substack{r\leq R\\(b,r)=1}}\e(-mb/r)\int_{-1/S}^{1/S} \Big(\sum_{n\leq x} f(n)\e(nb/r+(n-m)\beta)\Big)\d \beta 
\cr
= & \sum_{\substack{r\leq R\\(b,r)=1}}\sum_{\substack{ k\mid r}}\frac{\mu(r/k) k }{
\phi(r )}\sum_{(b,r)=1}\e(-mb/r)\int_{-1/S}^{1/S} \Big(\sum_{\substack{n\leq x\\ k\mid n}}
f(n) \e ((n-m)\beta) \Big)\d \beta 
\\ &\hspace{2cm}+O\Big( \Big(\frac 1S+\frac x{S^2}\Big) \frac{R^3 U_0(x)}{L(x)} \Big)  .
\end{split}\end{equation}

Combining Hypothesis \ref{hyphd} with summation by parts gives that
\begin{multline*}\sum_{\substack{n\leq x\\ k\mid n}}
f(n)\e ((n-m)\beta)= \sum_{0\leq j< J}  \frac{h_j(k)}{k} \int_0^x
u_j(t)\e ((t-m)\beta) dt
+O\Big(  (1+|\beta| x)  \max_{t\leq x} \frac{U_0(t)}{L(t)(\log Q(t))^2} \Big).
\end{multline*}

Applying the Euler-Maclaurin Formula to the main term in this expression and  using the notation 
$u_j^*(x):=|u_j(1)|+|u_j(x)|$
then implies that
\begin{multline*}\sum_{\substack{n\leq x\\ k\mid n}}
f(n)\e ((n-m)\beta)= \sum_{0\leq j< J}  \frac{h_j(k)}{k} \sum_{ n\leq x } 
u_j(n)\e ((n-m)\beta)
\\+O\Big( \sum_{0 \leq j < J}\frac{|h_j(k)|}{k} \big( u_j^*(x)+|\beta| U_0(x) \big)+(1+|\beta| x)\max_{t\leq x} \frac{U_0(t)}{L(t)(\log Q(t))^2} \Big),
\end{multline*}
where this time the main term on the right hand side should be small when $\|\beta\|$ is not.


Now, if $u_j$ is monotonic and differentiable, we have
\begin{align*} \sum_{ n\leq x }
u_j(n)\e (n\beta)&=u_j(x)E_1(x;\beta)-\int_1^xu_j'(t)E_1(t;\beta)\d t
\cr&\ll \frac{1}{||\beta||}\Big( |u_j(x)|+\int_1^x|u_j'(t)|\d t\Big)
 \ll \frac{u_j^*(x)}{||\beta||} .
\end{align*}
Hence, from \eqref{f(m)1} and introducing the Ramanujan sums
\begin{equation}\label{defcqm}
c_r(m):=\sum_{\substack{1\leq a\leq r\\(a,r)=1}}\e(ma/r)=\mu\Big(\frac{r}{(r,m)}\Big)\frac{\phi(r)}{\phi((r/(r,m))},
\end{equation} we obtain
\begin{equation}\label{f(m)2}\begin{split}
f(m) &
\approx\sum_{0\leq j< J}   \sum_{\substack{r\leq R\\k\mid r}}\frac{\mu(r/k) h_j(k) }{
\phi(r )}\sum_{(b,r)=1}\e(-mb/r)\int_{-1/S}^{1/S} \Big(\sum_{n\leq x}
u_j(n) \e ((n-m)\beta) \Big)\d \beta 
\cr &
\approx\sum_{0\leq j< J} \sum_{\substack{r\leq R }}\frac{(h_j*\mu)(r) }{
\phi(r )}c_r(m)\int_{-1/2}^{1/2} \Big(\sum_{n\leq x}
u_j(n) \e ((n-m)\beta) \Big)\d \beta =F_R(m),
\end{split}\end{equation}
with
\begin{equation}\label{deffQm}\begin{split}
F_R(m)&:=\sum_{0\leq j< J}u_j(m)\sum_{ r\leq R }\frac{(h_j*\mu)(r)  }{
\phi(r/(r,m))}\mu\Big(\frac{r}{(r,m)}\Big)\cr
&=  \sum_{0\leq j< J}u_j(m)\sum_{\substack{r\leq R}}\frac{(h_j*\mu)(r) }{
\phi(r )}\sum_{(b,r)=1}\e(mb/r).\end{split}
\end{equation}
 
Note that if we take $J=1$, $h(n)=\delta(n)$ (the neutral element in convolution) and $u_0 =1$, then we recover Vaughan's original approximation.

\section{First moment} 

The goal of this section is to prove Theorem \ref{thmoment1}, that is to bound the first moment
\begin{equation}\label{defM1}M_{1}(x,N, R;a):=\sum_{x/N<q\leq x} \big(\A^*(x;q,a)-\rho^*_R(x;q,a)\big).
\end{equation}
Similar arguments will yield an upper bound for
\begin{equation}\label{deftildeM1}\widetilde M_{1}(x,N, R;a):=\sum_{1\leq q\leq N} \big(\A^*(x;q,a)-\rho^*_R(x;q,a)\big).
\end{equation}
We denote
\begin{equation}\label{defsigmaR} \sigma_j(R):=\sum_{\substack{ r\leq R}}\frac{|(h_j*\mu)(r) |r }{
\phi(r )} .
\end{equation}

\begin{proof}[Proof of Theorem \ref{thmoment1}] We concentrate on bounding $M_{1}(x,N, R;a)$, since the proof of the corresponding bound for  $\widetilde M_{1}(x,N, R;a) $ is similar. We begin by estimating 
$$M_{11}(x,N;a):=\sum_{x/N<q\leq x} \A^*(x;q,a).$$
Under the condition $-x/N < a \leq x$, we can use the symmetry of the divisors $n-a=qs$ to obtain the identity
\begin{align*}
M_{11}(x,N;a)&=\sum_{1\leq s<N-aN/x}\sum_{\substack{a+sx/N<n\leq x\\n\equiv a\bmod s}}f(n) -\sum_{x+a<n\leq x}  f(n)
\cr 
&=\sum_{1\leq s<N-aN/x}
\big(\A^*(x;s,a)- \A^*(a+sx/N;s,a)\big)-\sum_{x+a<n\leq x}  f(n)  .
\end{align*}
Note that the last sum is empty when $a\geq 0$, and comes from the contribution of the case $s=1$, $q=n-a>x$. Applying Corollary \ref{corBV}, we obtain uniformly for $N+1\leq Q(x)$ that
\begin{equation}\label{estM11}\begin{split}M_{11}(x,N;a)&=\sum_{0\leq j< J} \sum_{1\leq s\leq N-aN/x }g_j(s,a)\int^x_{a+sx/N} u_j(t)\d t\cr&\quad -\sum_{0\leq j<J}\int_{x+a<t\leq x}u_j(t)\d t+O\Big(\tau(a)\frac{U_0(x)}{L(x)}\Big).\end{split}\end{equation}

We now give an estimate for 
\begin{align*}M_{12}(x,N, R;a)&=\sum_{x/N<q\leq x} \rho^*_R(x;q,a).\end{align*}
We have
\begin{equation}\label{M12}\begin{split}
M_{12}(x,N, R;a)&=\sum_{1\leq s<N-aN/x}\,\,\sum_{\substack{a+sx/N<n\leq x\\n\equiv a\bmod s}}F_R(n)-\sum_{x+a<n\leq x}  F_R(n)
\cr
&=\sum_{1\leq s<N-aN/x}
\big(\rho^*_R(x;s,a)- \rho^*_R(a+sx/N;s,a)\big)-\sum_{x+a<n\leq x}  F_R(n),
\end{split}\end{equation}
with
\begin{align*}\rho^*_R(x;s,a) &=  \sum_{\substack{a< n\leq x\\n\equiv a\bmod s}}\sum_{0\leq j< J} 
u_j(n)\sum_{ r\leq R }\frac{(h_j*\mu)(r)  }{
\phi(r/(r,n))}\mu\Big(\frac{r}{(r,n)}\Big)
\\&= \sum_{0\leq j< J}  \sum_{\substack{ a<n\leq x\\n\equiv a\bmod s}}
u_j(n)\sum_{ r\leq R }\frac{(h_j*\mu)(r)  }{
\phi(r )}\sum_{\substack{1\leq c\leq r\\(c,r)=1}}\e(nc/r)  .\end{align*}
Moreover, for $a<0$
$$\sum_{x+a<n\leq x}  F_R(n)=\rho^*_R(x;1,0)-\rho^*_R(x+a;1,0).$$
We now show that in $M_{12}(x,N, R;a)$, the contribution of the $r$ with $r\nmid s$ is small. Indeed, for fixed $j$ this contribution is given by
\begin{align*}&\sum_{\substack{ r\leq R \\r\nmid s}}\frac{(h_j*\mu)(r)  }{
\phi(r )}\sum_{\substack{1\leq c\leq r\\(c,r)=1}}\sum_{\substack{ a<n\leq x\\n\equiv a\bmod s}}u_j(n)\e(nc/r) \cr=&
\sum_{\substack{ r\leq R \\r\nmid s}}\frac{(h_j*\mu)(r)  }{
\phi(r )}\sum_{\substack{1\leq c\leq r\\(c,r)=1}}\e(ac/r)\sum_{\substack{ 1\leq k\leq (x-a)/s }}u_j(a+ks)\e(  cks /r). \end{align*}
Since each $u_j$ is monotonic and differentiable, we can show that the exponential sum is small (since $s/r$ is not an integer). Indeed, summation by parts gives
$$\sum_{\substack{ 1\leq k\leq (x-a)/s }}u_j(a+ks)\e(  cks /r)
\ll \frac{u_j^*(x) }{||cs/r||}+\frac{s}{||cs/r||}\int_{0}^{(x-a)/s}|u_j'(a+st)|\d t
\ll \frac{u_j^*(x)}{||cs/r||},$$
where 
$$u_j^*(x):=|u_j(x)|+1,$$ and $|| t ||$ denotes the distance to the nearest integer. The contribution of the terms with $r\nmid s$ is therefore
\begin{align*}&\ll \sum_{0\leq j< J}u_j^*(x)\sum_{1\leq s<N-aN/x}\sum_{\substack{ r\leq R \\r\nmid s}}\frac{|(h_j*\mu)(r) | }{\phi(r )}r \log r \cr&\ll  N(\log R)u_1^*(x,R).
\end{align*} 
 
 Hence,
\begin{align*}M_{12}(x,N, R;a) &=\sum_{0\leq j< J}\sum_{1\leq s<N-aN/x}
\sum_{\substack{ r\leq R \\r\mid s}}
\frac{(h_j*\mu)(r)  }{
\phi(r/(r,a))}\mu\Big(\frac{r}{(r,a)}\Big)
  \sum_{\substack{ x/N<k\leq (x-a)/s }}u_j(a+ks)\cr&\quad  -
\sum_{0\leq j< J}\sum_{x+a<n\leq x }u_j(n)
 +O\big(  N(\log R)u_1^*(x,R)\big)
\\
&=\sum_{0\leq j< J}\sum_{1\leq s<N-aN/x}\frac{1}{s}
\sum_{\substack{ r\leq R \\r\mid s}}
\frac{(h_j*\mu)(r)  }{
\phi(r/(r,a))}\mu\Big(\frac{r}{(r,a)}\Big)
  \int_{a+sx/N}^x u_j(t)\d t\cr&\quad -\sum_{0\leq j< J}\int_{x+a<t\leq x }u_j(t)\d t+O\big(  N(\log R)u_1^*(x,R)\big)\end{align*}
Indeed, summation by parts gives
\begin{equation}\label{som=int}\sum_{\substack{0< k\leq (x-a)/s }}u_j(a+ks)=\frac{1}{s}\int_a^x u_j(t)\d t+O(u_j^*(x)) 
\end{equation}
and for $|a|\leq x$,
$$\sum_{1\leq s<N-aN/x} 
\sum_{\substack{ r\leq R \\r\mid s}}
\frac{|(h_j*\mu)(r)|  }{
\phi(r/(r,a))}\mu^2\Big(\frac{r}{(r,a)}\Big)\ll 
\sum_{\substack{ r\leq R  }} \left[ \frac {2N} r \right]
\frac{|(h_j*\mu)(r)|  }{
\phi(r/(r,a))}  \ll 
N\sigma_j(R).$$
Substituting this in \eqref{M12}, we obtain an analogue of  \cite[Lemma 2.4(i)]{F15}
\begin{equation}\label{estM12}M_{12}(x,N, R;a)=m_{12}(R,N;a)-\sum_{0\leq j< J}\int_{x+a<t\leq x }u_j(t)\d t+O\Big( N(\log R)u_1^*(x,R)\Big)\end{equation}
with
\begin{equation}m_{12}(R,N;a):=\sum_{0\leq j< J}
\sum_{1\leq s\leq N -aN/x }\frac{1}{s}\Big(
\int_{a+sx/N}^x u_j(t)\d t \Big)
\sum_{\substack{ r\leq R \\r\mid s}}
\frac{(h_j*\mu)(r)  }{
\phi(r/(r,a))}\mu\Big(\frac{r}{(r,a)}\Big).
\end{equation}

We now approximate $m_{12}(R,N;a)$ with $m_{12}(\infty,N;a)$.
Let us show that
$$
\sum_{\substack{  r\mid s}}
\frac{(h_j*\mu)(r)  }{
\phi(r/(r,a))}\mu\Big(\frac{r}{(r,a)}\Big)=sg_j(s,a)
= \frac{s/(s,a)}{\phi(s/(s,a))} \sum_{d\mid s/(s,a)} \mu(d) \frac{h_j((s,a)d)}{d}  $$ for any arithmetical function  $h_j$.\footnote{Note that if $h_j$ is multiplicative, then the  function $s\mapsto g_j(s,a)$ defined in \eqref{eqdefgaq} is also multiplicative, and can easily compute its value at prime powers
 }
Opening the convolution, we have
\begin{align*}
\sum_{\substack{  r\mid s}}
\frac{(h_j*\mu)(r)  }{
\phi(r/(r,a))}\mu\Big(\frac{r}{(r,a)}\Big)&=
\sum_{\substack{  r\mid s}}
\frac{\mu(r/(r,a))}{
\phi(r/(r,a))}\sum_{d\mid r}h_j(d)\mu(r/d) 
\\
&= \sum_{d\mid s}h_j(d)
\sum_{\substack{    r'\mid s/d}}
\frac{\mu(r'd/(r'd,a))\mu(r')}{
\phi(r'd/(r'd,a))} .
\end{align*}
The innermost sum is empty as soon as $(s,a)\nmid d$.
We deduce that
\begin{align*}
\sum_{\substack{  r\mid s}}
\frac{(h_j*\mu)(r)  }{
\phi(r/(r,a))}\mu\Big(\frac{r}{(r,a)}\Big) &= \sum_{d'\mid s/(s,a)}h_j(d'(s,a))
\sum_{\substack{    r'\mid s/d'(s,a)}}
\frac{\mu({r'd'} )\mu(r')}{
\phi(r'd')} \\&
= \frac{s/(s,a)}{\phi(s/(s,a))}\sum_{d'\mid s/(s,a)}   \mu(d') \frac{h_j((s,a)d')}{d'},  
\end{align*}
which implies the desired identity.

We have
$$m_{12}(\infty,N;a)=\sum_{0\leq j< J} 
\sum_{1\leq s\leq N -aN/x }g_j(s,a)
\int_{a+sx/N}^x u_j(t)\d t .$$
Hence, for $R\geq N+1\geq N -aN/x$ we have
$$m_{12}(R,N;a)= m_{12}(\infty,N;a).$$ 
Therefore, substituting \eqref{estM11} and 
\eqref{estM12} in \eqref{defM1}, we have proved the desired bound on $M_{1}(x,N, R;a)$.

The bound on $\widetilde M_{1}(x,N, R;a)$ follows along similar lines. 
\end{proof}

\goodbreak

\section{Variance}

In this section we study the variance
\begin{equation}\label{defVxNRq} V(x,N,R;q):=\sum_{x/{2N} < q \leq x/N }\sum_{\substack{ 1\leq a \leq q}} \big(\A (x;q,a)-\rho _R(x;q,a)\big)^2=m(2N)-m(N), 
\end{equation}
where
$$ m(N)=\sum_{x/{N} < q \leq x }\sum_{\substack{  1\leq a \leq q }} \big(\A (x;q,a)-\rho _R(x;q,a)\big)^2.  $$

\begin{proof}[Proof of Theorem \ref{thmoment2}]
We have
\begin{align*}
m(N)&=\sum_{x/{N} < q \leq x}\sum_{\substack{  m,n \leq x \\ m\equiv n \bmod q }} \Delta_R(n)\Delta_R(m)\\ &=(\lceil x\rceil -\lfloor x/N\rfloor)  \sum_{ n\leq x} \Delta_R(n)^2  +2\sum_{x/{N} < q \leq x}\sum_{\substack{  m,n \leq x \\ m=n+qs \\ s\geq 1}} \Delta_R(n)\Delta_R(m). 
\end{align*} 

The second of these sums equals
$$\sum_{\substack{ n\leq x }}\Delta_R(n)M_{1}(x,N, R;n),$$
and hence we can apply Theorem \ref{thmoment1} and obtain the bound
$$ \ll  \sum_{\substack{n\leq x }}|\Delta_R(n)| \tau(n)\frac{U_0(x)}{L(x)} \ll \frac{U_0(x)x^{\frac 12}(\log x)^{\frac 32}}{L(x)} \|\Delta_R\|_2.
 $$
Hence, for $N\leq \min \{ Q(x), R, U_0(x)/(L(x)u_1^*(x,R) (\log R) ) \}$, we have the estimate \eqref{estVDelta}.

It is now clear that the variance is dominated by the diagonal terms. Let us evaluate these terms. We have
\begin{equation}\label{M2DELTA}
\|\Delta_R\|_2^2 = \|f\|_2^2 +\|F_R\|_2^2 -2 \sum_{ n\leq x} f(n)F_R(n). 
\end{equation}

For the third term on the right hand side of \eqref{M2DELTA}, we have by \eqref{deffQm} that
\begin{align*}
\sum_{ n\leq x} f(n) F_R(n)&=\sum_{0\leq j< J}\sum_{ n\leq x}
u_j(n)f(n)\sum_{ r\leq R }\frac{(h_j*\mu)(r)  }{
\phi(r/(r,n))}\mu\Big(\frac{r}{(r,n)}\Big)\cr
&=\sum_{0\leq j< J}\sum_{ r\leq R }(h_j*\mu)(r)\sum_{k\mid r}\frac{ \mu ( {r}/k) }{
\phi(r/k)}\sum_{\substack{ n\leq x\\(n,r)=k}}
u_j(n)f(n)
\cr
&=\sum_{0\leq j< J}\sum_{ r\leq R }(h_j*\mu)(r)\sum_{k\ell\mid r}\frac{ \mu ( {r}/k)\mu(\ell) }{
\phi(r/k)}\sum_{\substack{ n\leq x\\ k\ell\mid n}}
u_j(n)f(n)
\cr
&=\sum_{0\leq j< J}\sum_{ r\leq R }(h_j*\mu)(r)\sum_{m\mid r}\kappa(m,r)\sum_{\substack{ n\leq x\\ m\mid n}}
u_j(n)f(n)
\end{align*}
where for $m\mid r$ we put
$$\kappa(m,r):=\sum_{k\ell=m}\frac{ \mu ( {r}/k)\mu(\ell) }{
\phi(r/k)}.$$
A calculation gives
$$\kappa(m,r)=\frac{ \mu ( {r}/m)  }{
\phi(r/m)}\prod_{\substack{p\mid m\\ v_p(m)=v_p(r) }}(1-1/p)^{-1}.$$
Applying integration by parts, Hypothesis \eqref{hyphd} gives that for all $j$, 
\begin{align*}
\sum_{\substack{ n\leq x\\ m\mid n}}
u_j(n)f(n)=
\sum_{0\leq j'<J}\frac{h_{j'}(m)}{m}\int_{0}^{x}
u_j(t)u_{j'}(t)\d t+O\left( \frac{U_0(x)u_j^*(x)}{L(x)(\log Q(x))^2}\right).
\end{align*}
However, \begin{equation}\label{majsumkappa}
\sum_{m\mid r}|\kappa(m,r)|\ll   \Big(\frac{r}{\phi(r)}\Big)^2,
\end{equation}
and hence
\begin{equation}\label{sumfF}\begin{split}\sum_{ n\leq x} f(n) F_R(n)&=\sum_{0\leq j ,j'< J}
  \sum_{ r\leq R }(h_j*\mu)(r)\sum_{m\mid r}\kappa(m,r)\frac{h_{j'}(m)}{m}\int_{0}^{x}
u_j(t)u_{j'}(t)\d t\\ &\quad +O \left( \frac{U_0(x)u_2^*(x,R)}{L(x)(\log Q(x))^2}\right).\end{split}
\end{equation} 

A calculation shows that
$$\sum_{m\mid r}\kappa(m,r)\frac{h_{j'}(m)}{m}=\frac{(h_{j'}*\mu)(r)}{\phi(r)}  .$$
For the second term on the right hand side of \eqref{M2DELTA}, by \eqref{deffQm} we have
$$\|F_R\|_2^2 = \sum_{0\leq j ,j'< J}\sum_{ n\leq x}
u_j(n)u_{j'}(n)\sum_{ r\leq R }\sum_{ r'\leq R }\frac{(h_j*\mu)(r)(h_{j'}*\mu)(r')  }{
\phi(r/(r,n))\phi(r'/(r',n))}\mu\Big(\frac{r}{(r,n)}\Big)\mu\Big(\frac{r'}{(r',n)}\Big) .$$
Analogous steps give the identity \begin{align*}
\|F_R&\|_2^2  =\sum_{0\leq j ,j'< J}\sum_{\substack{ r\leq R \\ r'\leq R }} (h_j*\mu)(r)(h_{j'}*\mu)(r')\sum_{\substack{k\mid r\\k'\mid r'}}\frac{ \mu ( {r}/k) }{
\phi(r/k)}\frac{ \mu ( {r'}/k') }{
\phi(r'/k')}\sum_{\substack{ n\leq x\\(n,r)=k\\(n,r')=k'}}
u_j(n)u_{j'}(n)\\
&=\sum_{0\leq j ,j'< J}\sum_{\substack{ r\leq R \\ r'\leq R }} (h_j*\mu)(r)(h_{j'}*\mu)(r')\sum_{\substack{k\ell\mid r\\k'\ell'\mid r'}}\frac{ \mu ( {r}/k)\mu ( {r'}/k') \mu(\ell)\mu(\ell')}{
\phi(r/k) 
\phi(r'/k')}  \sum_{\substack{ n\leq x\\ [k\ell,k'\ell']\mid n}}
u_j(n)u_{j'}(n)
\\
&=\sum_{0\leq j ,j'< J}\sum_{\substack{ r\leq R \\ r'\leq R }} (h_j*\mu)(r)(h_{j'}*\mu)(r')\sum_{\substack{m\mid r\\ m'\mid r'}}\kappa(m ,r ) \kappa(m',r') \sum_{\substack{ n\leq x\\ [m,m']\mid n}} 
u_j(n)u_{j'}(n)
\end{align*}
The estimate \eqref{som=int} applied to $u_j u_{j'}$ gives that
$$\sum_{\substack{ n\leq x\\ [m,m']\mid n}} 
u_j(n)u_{j'}(n)=\frac1{[m,m']}\int_{0}^{x}
u_j(t)u_{j'}(t) \d t+O\big(u_j^*(x)u_{j'}^*(x)\big)$$
 
We obtain
\begin{align*}\|F_R\|_2^2=&\sum_{0\leq j ,j'< J}
  \sum_{\substack{ r\leq R \\ r'\leq R }} (h_j*\mu)(r)(h_{j'}*\mu)(r')\sum_{\substack{m\mid r\\ m'\mid r'}}\frac{\kappa(m ,r ) \kappa(m',r') }{[m,m']}\int_{0}^{x}
u_j(t)u_{j'}(t)\d t\cr&+O\big(\sigma_{R,2}^2  \big),\end{align*}
where, by \eqref{majsumkappa},
\begin{align*}\sigma_{R,2}&:=\sum_{0\leq j  < J}\sum_{\substack{ r\leq R  }} |(h_j*\mu)(r)|\sum_{\substack{m\mid r }}|\kappa(m ,r )|u_j^*(x) \cr&\leq 
\sum_{0\leq j  < J}\sum_{\substack{ r\leq R  }} |(h_j*\mu)(r)| \Big(\frac{r}{\phi(r)}\Big)^2u_j^*(x) 
=  u_2^*(x,R) .
\end{align*} 
A calculation shows that
$$\sum_{\substack{m\mid r\\ m'\mid r'}}\frac{\kappa(m ,r ) \kappa(m',r') }{[m,m']}=\frac{1_{r=r'}}{\phi(r)}.$$
Hence, we have that
\begin{equation}\label{sumFF}\|F_R\|_2^2=
  \sum_{\substack{ r\leq R  }}  \frac{1}{\phi(r)} \int_{0}^{x}
\Big(\sum_{0\leq j  < J}(h_j*\mu)(r) u_j(t)\Big)^2\d t+O\big( u_2^*(x,R)^2  \big).
\end{equation} 
By substituting \eqref{sumfF} and \eqref{sumFF} in \eqref{M2DELTA}, we obtain \eqref{estM2DELTA}.
\end{proof}

\section{Sum of divisors}

%
%

We now consider the $k$-th divisor function $f(n)=\tau_k(n)$, and show that Hypotheses \ref{hyphd} and \ref{hypequir} hold. To describe the implied arithmetical functions $h_{j,k}(d)$ we will write
$$F_{d,k}(s):=\sum_{n=1}^\infty \frac{\tau_k(dn)}{\tau_k(d)n^s}=\zeta(s)^k G_{d,k}(s),$$
where
$$G_{d,k}(s)=\prod_{p\mid d } \Big( 1-\frac{1}{p^s}\Big)^k \Big(1+\sum_{\mu\geq 1}
\frac{\tau_k(p^{\mu+v_p(d)})}{\tau_k(p^{v_p(d)})p^{\mu s}}\Big),$$ and $v_p$ denotes the $p$-adic valuation. 
 The arithmetic functions $h_{j,k}$ in Hypothesis \ref{hyphd} will be of the form 
 
 \begin{equation}
  \qquad\qquad h_{j,k}(d)=\tau_k(d)\Big( \frac{b_{j,k}(d)}{ j!}+\frac{b_{j-1,k}(d)}{ (j-1)!}\Big)\qquad\qquad ( 0\leq j\leq k),
  \label{eqdefhjk}
\end{equation}  
 where
\begin{equation}
 b_{j,k}(d):= \Big(\frac{\{\zeta(s)(s-1)\}^k G_{d,k}(s)}{sd^{s-1}}\Big)^{(j)}(1).
 \label{eqdefbjk}
\end{equation} 
(For ease of notation we set $b_{-1}(d)=0$.)

\begin{proposition}\label{proptauk} Let $f(n)=\tau_k(n)$, with $k\geq 2$.
Hypothesis \ref{hyphd} holds with the choice
$J=k,$  $u_{j,k}(t)=(\log t)^{k-j-1}/(k-1-j)!$, $h_{j,k}$ as in \eqref{eqdefhjk} and any function $L$ such that $L(x)\leq x^{2/(k+2)-\varepsilon} $. Moreover, we have the bound \begin{equation}|(h_{j,k}*\mu)(d)|\ll_k (\tau_{k-1}*|\tilde h_k|)(d) (\log d)^{j},
\label{eqmajhjkhtilde}
\end{equation} 
where $\tilde h_k$ is an arithmetical function such that
$$\sum_{d=1}^\infty \frac{|\tilde h_k(d)|}{d}\ll_k 1.$$
Finally, Hypothesis~\ref{hypequir} with the choice $Q(x)= x^{1/2-c/k}$ and any function $L$ such that $L(x)\leq x^{ c/k}$ for some suitably chosen absolute constant $0<c<1/2$.
\label{veriftau}
\end{proposition}

\begin{remark} 
\
\begin{enumerate}
\item In the case $k=3$ Heath-Brown showed \cite{HB86} that for any fixed $\varepsilon>0$ and  any $Q(x)\leq x^{11/21-\varepsilon}$ we can take $L(x) = x^{-\varepsilon} (x^{11/21}/Q)^{7/17}$ in Hypothesis \ref{hypequir}. Restricting the sum to prime moduli and taking $L(x)=(\log x)^B$, Fouvry, Kowalski and Michel  \cite{FKM15} showed that the choice $Q=x^{9/17-\varepsilon}$ is admissible (again for any fixed $\varepsilon>0$).

\item Drappeau \cite[Theorem 7.1]{D15} recently showed that there exists $\eta>0$ such that for $k\geq 4$, $|a|\leq x^\eta$ and $Q\leq x^{1/2+\eta}$, we have the bound
$$\sum_{\substack{q\leq Q\\ (q,a)=1}}\Bigg(\sum_{\substack{n\leq x \\ n\equiv a \bmod q }}\tau_k(n)-\frac{1}{\phi(q)} \sum_{\substack{n\leq x \\ (n, q)=1}}\tau_k(n)\Bigg)\ll x^{1-\eta/k}.$$
(Note that the summand has no absolute values.)
\end{enumerate}

\end{remark}

Before proving Proposition \ref{proptauk}, we need to estimate the summatory function of $\tau_k(dn)/\tau_k(d)$. This will be done using standard tools, and will allow us to deduce Hypothesis \ref{hyphd} for the sequence $f=\tau_k$.

\begin{lemma}
\label{lemperron}
We have
$$\sum_{n\leq x} \frac{\tau_k(dn)}{\tau_k(d) }=xP_d(\log x)+O_{\varepsilon}\Big(\tau(d)x^{1-\frac 2{k+2}+\varepsilon}\Big) ,$$
where $P_d(X)$ is the polynomial in $X=\log x$ with coefficients depending on $d$ defined by
$$P_d(X):={\rm Res}_{s=1}\Big(\zeta(s)^kG_{d,k}(s)\frac{\e^{X(s-1)}}{s} \Big)
=\sum_{j=0}^{k-1} a_{j,k}(d)\frac{X^{k-1-j}}{j!(k-j-1)!}.$$

\end{lemma}
\begin{proof}

Fix $\varepsilon > 0$. First note that for $\Re (s) \geq \varepsilon>0$ we have the bound
$$ G_{d,k}(s) \ll_{\varepsilon} \tau(d).  $$
One easily shows that $\tau_k(dn) \leq \tau_k(d)\tau_k(n)$, and hence the coefficients of the Dirichlet series we are interested in are bounded by $\tau_k(n)$. An effective Perron formula gives the estimate
$$ \sum_{n\leq x} \frac{\tau_k(dn)}{\tau_k(d) }=\frac 1{2\pi i} \int_{1+(\log x)^{-1}-iT}^{1+(\log x)^{-1}+iT} \zeta(s)^k G_{d,k}(s) x^s \frac{\d s}{s}+O_{\varepsilon}\Big( \frac {x^{1+\varepsilon}}T  \Big). $$ 
Shifting the contour of integration to the left we uncover a pole at $s=1$, and end up with an integral on the line  $\Re(s) = \varepsilon$, which is
$$ \ll_{\varepsilon} \tau(d) x^{\varepsilon} T^{\frac k2}.  $$
The result follows by taking $T= x^{\frac 2{k+2}}$.
\end{proof}

We now achieve the verification of the hypotheses for the sequence $f=\tau_k$.
\begin{proof}[Proof of Proposition \ref{proptauk}] 
To show Hypothesis \ref{hypequir}, we see that from the work of Bombieri, Friedlander and Iwaniec \cite{BFI86} we can deduce the existence of constants $c'>0$ and $A_k$ such that for any fixed $k\geq 1$ we have $$\sum_{q\leq Q}\max_{y\leq x} \max_{\substack{1\leq a\leq q\\ (a,q)=1}}\Bigg|\sum_{\substack{n\leq y \\ n\equiv a \bmod q }}\tau_k(n)-\frac{1}{\phi(q)} \sum_{\substack{n\leq y \\ (n, q)=1}}\tau_k(n)\Bigg|\ll x^{1-c'/k}+Q\sqrt{x}(\log x)^{A_k}.$$
Taking $Q=x^{1/2-c/k}$, this is
$\ll_k x^{1-c'/k} (\log x)^{A_k}$. This bound can be extended to values of $a$ that are not necessarily coprime to $q$  (note also that it is trivial for $a=0$), and hence Hypothesis~\ref{hypequir} holds with the choice $L(x)=x^{ c'/k} (\log x)^{-A_k}\gg x^{ c /k} $, for any $0<c <c'$. Note that Hypothesis~\ref{hypequir} also follows from \cite{M76}. 

We now proceed to the verification of Hypothesis \ref{hyphd}. A calculation shows that the coefficients of Lemma \ref{lemperron} satisfy
$$a_{j,k}(d)=\Big(\frac{\{\zeta(s)(s-1)\}^k G_{d,k}(s)}{s}\Big)^{(j)}(1).$$
Hence, by the same Lemma we have that for $d\leq x$
\begin{equation}
\label{eqAd}
 \A_d(x) =\sum_{\substack{n\leq x \\ d\mid n}} \tau_k(n) = \frac{\tau_k(d)}{d}xP_d(\log x-\log d) + O_{\varepsilon}\Big(x^{\varepsilon}\Big(\frac{x}d \Big)^{1-\frac 2{k+2}}\Big).
\end{equation} 
 Expanding the main term in this expression gives that
\begin{align*}P_d(\log x&-\log d)
 = \sum_{j=0}^{k-1} a_{j,k}(d)\frac{(\log x-\log d)^{k-1-j}}{j!(k-j-1)!}\\
&= \sum_{j=0}^{k-1} 
\frac{a_{j,k}(d)}{j!(k-j-1)!}\sum_{m=0}^{k-j-1}\Big(\mycom{k-j-1}{m}\Big)(\log x)^{k-1-j-m}(-\log d)^{m}\\
&=\sum_{\ell=0}^{k-1} b_{\ell,k}(d)\frac{(\log x)^{k-1-\ell}}{ \ell!(k-\ell-1)!},
\end{align*}
where the $b_{\ell,k}(d)$ are defined in \eqref{eqdefbjk}. Indeed a calculation shows that
$$b_{j,k}(d)= \sum_{\ell=0}^j\Big(\mycom{j}{\ell}\Big)
a_{\ell,k}(d)(-\log d)^{ j-\ell }.$$

Hence, 
\begin{align*} xP_d(\log x-\log d)&=\int_0^x \Big(\sum_{\ell=0}^{k-2} \frac{b_{\ell,k}(d)}{\ell!}\Big[\frac{(\log t)^{k-1-\ell}}{(k-\ell-1)!}+\frac{(\log t)^{k-2-\ell}}{(k-\ell-2)!} \Big] +\frac{b_{k-1,k}(d)}{(k-1)!}  \Big) \d t \\
&=  \sum_{j=0}^{k-1}  \Big[\frac{b_{j,k}(d)}{j!}+\frac{b_{j-1,k}(d)}{(j-1)!} \Big]\int_0^x \frac{(\log t)^{k-1-j}}{(k-1-j)!}  \d t.
\end{align*}
 Substituting this in \eqref{eqAd}, it follows that Hypothesis \ref{hyphd} holds with the choice $J=k,$
\begin{equation} u_{j,k}(t)= \frac{(\log t)^{k-1-j}}{(k-1-j)!};
\label{eqdefujk} \end{equation}
and $h_{j,k}(d)$ as defined in \eqref{eqdefhjk}. 
In particular we have
\begin{equation}\label{calculh0d}
h_{0,k}(d)=\tau_k(d)a_{0,k}(d)=\tau_k(d)G_{d,k}(1).\end{equation}

The last thing to show is the bound on $|(h_{j,k}*\mu)(d)|$. This will be done in Lemma \ref{lemmajhj}.
\end{proof}

\begin{lemma}
\label{lemmajhj}
Let $h_{j,k}$ be defined as in \eqref{eqdefhjk}. There exists an arithmetical function 
$\tilde h_k$ such that that the following two bounds hold:
\begin{equation}
\label{sumtildeh}
\sum_{d=1}^\infty \frac{|\tilde h_k(d)|}{d}\ll_k 1;
\end{equation}
\begin{equation}\label{majhj*muq}|(h_{j,k}*\mu)(d)|\ll_k (\tau_{k-1}*|\tilde h_k|)(d) (\log d)^{j}.
\end{equation}
\end{lemma}
\begin{proof}

By definition of the $h_{j,k}$ (see \eqref{eqdefhjk}), we have that
\begin{equation}\label{defnewhjk}\begin{split}h_{j,k}(d)
&=
\frac{\tau_k(d)}{j!}\Big(\frac{\{\zeta(s)(s-1)\}^k G_{d,k}(s)}{sd^{s-1}}\Big)^{(j)}(1)\cr&\quad
+\frac{\tau_k(d)}{(j-1)!}\Big(\frac{\{\zeta(s)(s-1)\}^k G_{d,k}(s)}{sd^{s-1}}\Big)^{(j-1)}(1)
.
\end{split}\end{equation}
To tackle the size of $h_{j,k}*\mu$, we study the associated Dirichlet series. We have that
\begin{align*}
\sum_{d=1}^\infty \frac{h_{j,k}(d)}{d^z}&=
\frac{1}{j!}\Big(\frac{\{\zeta(s)(s-1)\}^k }{s }\sum_{d=1}^\infty \frac{\tau_k(d)G_{d,k}(s)}{d^{z+s-1}}\Big)^{(j)}(s=1)\cr
&\quad +
\frac{1}{(j-1)!}\Big(\frac{\{\zeta(s)(s-1)\}^k }{s }\sum_{d=1}^\infty \frac{\tau_k(d)G_{d,k}(s)}{d^{z+s-1}}\Big)^{(j-1)}(s=1).\end{align*}
Defining $H_k(s,z)$ by 
\begin{equation}\label{defHk}\begin{split}
\sum_{d=1}^\infty \frac{\tau_k(d)G_{d,k}(s)}{d^{z+s-1}}&=\prod_{p}\Big(1+\sum_{\nu\geq 1}\frac{\tau_k(p^\nu)G_{p^\nu,k}(s)}{p^{\nu(z+s-1)}}\Big)\cr
&=\zeta(z+s-1)^kH_k(s,z)
,\end{split}\end{equation}
we see that $H_k(s,z)$ converges absolutely in the region $\Re(s)+\Re(z)>\tfrac 32$ (this holds for example when $\Re e (s)>\tfrac 34$ and $\Re e (z)>\tfrac 34$), 
hence
$$
 \Big( \sum_{d=1}^\infty \frac{\tau_k(d)G_{d,k}(s)}{d^{z+s-1}}\Big)^{(j)} =
\sum_{r}c(j_0,j_1, \ldots,j_r) \big(\zeta^{(j_1)}\cdots \zeta^{(j_r)}\zeta^{k-r}\big)(z+s-1)H_k(s,z)^{(j_0)},$$
where the variables in the sum satisfy $j_1+\ldots+j_r+j_0=j\leq k-1$ and $r \leq j$. Hence
\begin{equation}
 \zeta(z)^{-1}\Big( \sum_{d=1}^\infty \frac{\tau_k(d)G_{d,k}(s)}{d^{z+s-1}}\Big)^{(j)}(s=1) =
\sum_{r}c(j_0, \ldots,j_r) \big(\zeta^{(j_1)}\cdots \zeta^{(j_r)}\zeta^{k-r-1}\big)(z )H_k(1,z)^{(j_0)}.
\label{eqdefgjk}
\end{equation}

Define $\tilde h_{k}$ to be the arithmetical function associated to the Dirichlet series $H_k(1,z)$. To deduce the desired bound \eqref{majhj*muq}, we denote by  $g_{j,k}$ the arithmetical function associated to the Dirichlet series in \eqref{eqdefgjk}. We have
$$g_{j,k}=\sum_{r}c(j_0, \ldots,j_r) \big((-\log)^{j_1}*\cdots * (-\log)^{j_r}*\tau_{k-r-1}\big)*(\tilde h_{k }(-\log )^{j_0}),$$
Hence $$g_{j,k}(d)\ll (\log d)^j\sum_{r}  \big(\big(1*\cdots * 1*\tau_{k-r-1}\big)*|\tilde h_{k} |\big)(d)\ll
(\log d)^j  \big( \tau_{k -1} *|\tilde h_{k} |\big)(d).$$
Since $h_{j,k}*\mu$ is a linear combination of $g_{j',k}$ with $j'\leq j$, we need to show the bound \eqref{sumtildeh}.
We have
$$G_{p^\nu,k}(1)=\Big( 1-\frac{1}{p}\Big)^k \Big(1+\sum_{\mu\geq 1}
\frac{\tau_k(p^{\mu+\nu})}{\tau_k(p^{\nu})p^{\mu }}\Big)=1+O_k(1/p) $$
and hence, by \eqref{defHk},
$$1+\sum_{\mu\geq 1}\frac{\tilde h_{k}(p^\mu)}{p^{\mu z}}=
1+\Big(1-\frac{1}{p^z}\Big)^{k}
\sum_{\mu\geq 1}\frac{\tau_k(p^\mu) }{p^{\mu z}}\big(G_{p^\mu,k}(1)-1 \big)$$
which then implies that
$$\big|\tilde h_{k}(p^\mu)\big| \ll_k \tau_k(p^\mu)/p.$$
 The bound \eqref{sumtildeh} follows. 

\end{proof}

We now deduce the following corollary of Theorem \ref{thmoment1}, which implies Theorem \ref{thmomenttauk1}.
\begin{corollary}Let $k\geq 2$. If $f=\tau_k$, then the choices  in Proposition \ref{proptauk} give that for $a\neq 0$, $-x/N < a \leq x $, $N+1\leq \min\{x^{1/2-c/k}, R\}$, we have the bound
$$|M_{1}(x,N, R;a)|+|\widetilde M_{1}(x,N, R;a)|\ll  U_0(x)\Big\{\frac{NR(\log R)^{k}}{x}  +\frac{\tau(a)}{x^{c'/k} }\Big\},$$ where $c$ and $c'$ are two positive absolute constants coming from the second remark in Lemma~\ref{proptauk}.
\end{corollary} 
 
\begin{proof} 
We need to bound the quantities $\sigma_j(R)$ and $u_1^*(x,R)$ appearing in \eqref{defsigmaR} and \eqref{defu1*xR}. The bound~\eqref{majhj*muq} implies that
$$\qquad\qquad\qquad\qquad\sigma_j(R)\ll R(\log R)^{k+j -1}\qquad\qquad\qquad\qquad (0\leq j\leq k-1),$$
and hence
\begin{equation}\label{maju*xR}
u_1^*(x,R)\ll u_0(x)R(\log R)^{k-1}.
\end{equation}

\end{proof}

Note that we can deduce the following bound from the last proof:
\begin{equation}\label{precisionhj}
(h_{j,k}*\mu)(d)=\frac{1}{j!}\big((h_{0,k} (-\log  )^{j})*\mu\big)(d)
+O\Big( (\tau_{k-1}*|\tilde h|)(d) (\log d)^{j-1}\Big).
\end{equation}
Indeed the terms with $j_0\geq 1$ are all bounded by the error term. In a sense this means that up to this error term we can replace $G_{d,k}(s)$ by $G_{d,k}(1)$ in the definition \eqref{defnewhjk} of $h_{j,k}$.

\section{Variance of the sum of divisors function}

\label{section variance tauk}

In this section we apply Theorem \ref{thmoment2} to the function $f=\tau_k$ and evaluate the main terms. 

First note that there exists a polynomial $Q_k$ of degree $k^2-2$ such that
$$\|\tau_k\|_2^2 =C_k x\frac{(\log x)^{k^2-1}}{(k^2-1)!}\Big(1+\frac 1{\log x}Q_k\Big(\frac 1{\log x} \Big)+O\Big(\frac{1}{ x^{2/(k^2+2) -\varepsilon}}\Big)\Big),$$
where
$$ C_k := \prod_p \Big( 1-\frac 1p \Big)^{k^2}\Big( \sum_{\nu\geq 0 }\frac{  \tau_{k}(p^{\nu})^2}{ p^{\nu}}\Big) . $$
Each of the factors of this product is a polynomial in $1/p$. Indeed\footnote{See \cite[Page 7, Footnote 3]{RS16}},
$$ C_k  = \prod_p \Big( 1-\frac 1p \Big)^{(k-1)^2}\Big( \sum_{0\leq\nu\leq k-1 }\Big( \mycom{k -1}{ \nu}\Big)^2\frac{ 1}{ p^{\nu}}\Big) . $$

\begin{theorem}\label{thmomenttauk2} Let $k\geq 2$ and consider the arithmetical function $f=\tau_k$. There exists an absolute constant $c'>0$ such that for $R\leq x^{c'/k}$ we have
\begin{equation}\label{estM2taukDELTA} \begin{split}
\|\Delta_R\|_2^2 &= C_k x \frac{(\log x)^{k^2-1}}{(k^2-1)!} \Big(1-\Big(\frac{\log R}{\log x} \Big)^{(k-1)^2}P_k\Big(\frac{\log R}{\log x} \Big)+\frac 1{\log x} Q_k\Big( \frac 1{\log x}\Big)\Big) \\&\hspace{2cm}+ O\big( x(\log x)^{2k-2}(\log R)^{k^2-2k}\big), \end{split}
\end{equation}
where $P_k$ is a polynomial of degree $2k-2$ defined by
$$P_k(v):= \sum_{0\leq j,j' < k} \frac{(k^2-1)!c_{j,j'} v^{j+j'}}{(k-j-1)!(k-j'-1)!} ; \hspace{.2cm} c_{j,j'}:= {\rm Res}_{s=0}{\rm Res}_{\substack{s_1=0\\ s_2=0}}\frac{ \e^{s} (s+s_1)^k(s+s_2)^k  }{s_1^{j+1}s_2^{j'+1}s^2(s+s_1+s_2)^{k^2}}  .$$
\end{theorem}


\begin{remark} 
\
\label{remark following 6.1}
\begin{enumerate}

\item We have that $c_{0,0}=1/((k-1)^2)!$, hence the constant term of $P(v)$ equals
$$ \frac {(k^2-1)!}{((k-1)^2)!(k-1)!^2} \asymp k^{2k-3}\e^{2k}, $$
however when multiplied out by $(\log R/ \log x)^{(k-1)^2}$, this quantity is $\ll c^{-k^2}$ for some absolute constant $c$. Hence, since $R \leq x^{c''/k}$, the term in parentheses in \eqref{estM2taukDELTA} is strictly smaller than $1$.

\item When $N\leq R\ll x^{c''/k}$, one should be able to show that
$$\|\Delta_R\|_2^2=xQ_k(\log x,\log R)+O\Big(\frac x{R^{c'''/k^2}}\Big),$$
where $Q_k$ is a polynomial (not necessarily homogeneous) and $c'',c'''>0$ are suitably chosen absolute constants. Such an estimate would allow one to obtain more precise version of Theorem \ref{thmoment2tauk}, however we preferred to keep the simpler statement \eqref{estDeltatauk}.

\item These calculations should be compared to those of \cite{RS16}, where the expected main term is given by
\begin{align*}&=\tilde a_k\Big(\frac{c^{k^2-1}}{(k^2-1)!}+{\rm Res}_{z=0}{\rm Res}_{\substack{s_1=0\\ s_2=0}}\e^{z}\e^{c(s_1+s_2-z)}\frac{(s_1-z)^k(s_2-z)^k}{z^2s_1^ks_2^k(s_1+s_2-z)^{k^2}} \Big)
\\
&=\tilde a_k\Big(\frac{c^{k^2-1}}{(k^2-1)!}+c^{k^2-1}{\rm Res}_{z=0}{\rm Res}_{\substack{s_1=0\\ s_2=0}}\e^{-z(c-1)/c}\e^{ s_1+s_2 }\frac{(s_1-z)^k(s_2-z)^k}{z^2s_1^ks_2^k(s_1+s_2-z)^{k^2}} \Big)
\\
&=\tilde a_k\Big(\frac{c^{k^2-1}}{(k^2-1)!}-c^{k^2-1}{\rm Res}_{z=0}{\rm Res}_{\substack{s_1=0\\ s_2=0}}\e^{ z(c-1)/c}\e^{ s_1+s_2 }\frac{(s_1+z)^k(s_2+z)^k}{z^2s_1^ks_2^k(s_1+s_2+z)^{k^2}} \Big).\end{align*} 
Moreover for $c\in [1,2)$, Rodgers and Soundararajan proved \cite[Lemma 6]{RS16} that
$$\gamma_k(c)=
\frac{c^{k^2-1}}{(k^2-1)!}-c^{k^2-1}{\rm Res}_{z=0}{\rm Res}_{\substack{s_1=0\\ s_2=0}}\e^{ z(c-1)/c}\e^{ s_1+s_2 }\frac{(s_1+z)^k(s_2+z)^k}{z^2s_1^ks_2^k(s_1+s_2+z)^{k^2}}.
$$

 \par
We have
\begin{align*}\frac{P_k(v)}{(k^2-1)!}&= {\rm Res}_{s=0}{\rm Res}_{\substack{s_1=0\\ s_2=0}}  \frac{ v^{2k-2}\e^{s} (s+s_1)^k(s+s_2)^k  }{s_1^{k}s_2^{k}s^2(s+s_1+s_2)^{k^2}} \sum_{0\leq j,j' < k} \frac{  (s_1/v)^{k-j-1 }(s_2/v)^{k-j'-1 }}{(k-j-1)!(k-j'-1)!} \\
&=  {\rm Res}_{s=0}{\rm Res}_{\substack{s_1=0\\ s_2=0}}  \frac{ v^{2k-2}\e^{s+(s_1+s_2)/v} (s+s_1)^k(s+s_2)^k  }{s_1^{k}s_2^{k}s^2(s+s_1+s_2)^{k^2}}\\
&= v^{-(k-1)^2} {\rm Res}_{s=0}{\rm Res}_{\substack{s_1=0\\ s_2=0}}  \frac{ \e^{vs+ s_1+s_2 } (s+s_1)^k(s+s_2)^k  }{s_1^{k}s_2^{k}s^2(s+s_1+s_2)^{k^2}} .
\end{align*} 
From these calculations, we get
\begin{align}\frac{1-v^{(k-1)^2}P_k(v)}{(k^2-1)!}&
=(1-v)^{k^2-1}\gamma_k\Big(\frac{1}{1-v}\Big) 
,
\label{eqcalculgammak}
\end{align}
since in our domain we have $1/(1-v)\in [1,2).$
\end{enumerate}
\end{remark}

We will require the following calculation.
\begin{lemma}\label{Ck} We have
$$C_{0,0}(k):=\prod_p \Big(1-\frac 1p\Big)^{(k-1)^2}\Big( \sum_{\nu\geq 0}   \frac{ (h_{0,k}*\mu)(p^\nu)^2}{\phi(p^\nu)}\Big)=C_k.$$
\end{lemma}

\begin{proof}
The formula \eqref{calculh0d} implies that $$
h_{0,k}(d)= \tau_k(d)G_{d,k}(1)=
\prod_{p\mid d } \Big( 1-\frac{1}{p}\Big)^k \Big( \sum_{\mu\geq 0}
\frac{\tau_k(p^{\mu+v_p(d)})}{ p^{\mu }}\Big).$$  
However, for $\nu\geq 1$ we have
\begin{align*}(h_{0,k}*\mu)(p^\nu)&= \Big( 1-\frac{1}{p}\Big)^k\sum_{\mu\geq 0}
\frac{\tau_k(p^{\mu+\nu})-\tau_k(p^{\mu+\nu-1})}{ p^{\mu }}
\cr&= \Big( 1-\frac{1}{p}\Big)^k\sum_{\mu\geq 0}
\frac{\tau_{k-1}(p^{\mu+\nu})}{ p^{\mu }} =\Big( 1-\frac{1}{p}\Big)h_{0,k-1}(p^\nu),\end{align*}
since
\begin{align*}\tau_k(p^{\mu+\nu})-\tau_k(p^{\mu+\nu-1})&=(\tau_k*\mu)(p^{\mu+\nu}) =\tau_{k-1}(p^{\mu+\nu}).\end{align*}
Moreover $$(h_{0,k}*\mu)(1)=1
=\Big( 1-\frac{1}{p}\Big)^{k-1}\sum_{\mu\geq 0}
\frac{\tau_{k-1}(p^{\mu}) }{ p^{\mu }}.$$
Hence
\begin{equation}\label{som1tautau}
\sum_{\nu\geq 0}   \frac{  (h_{0,k}*\mu)(p^\nu)^2}{\phi(p^\nu)}
 =1+\Big( 1-\frac{1}{p}\Big)^{2k-1}\sum_{\substack{\nu\geq 1\\ \mu,\mu'\geq 0}}\frac{ \tau_{k-1}(p^{\mu+\nu})\tau_{k-1}(p^{\mu'+\nu})}{ p^{\nu+\mu+\mu'}}.\end{equation}
Fixing $\ell=\nu+\mu+\mu'$, we see that
\begin{align*}\sum_{\substack{\nu\geq 1\\ \mu,\mu'\geq 0}}\frac{ \tau_{k-1}(p^{\mu+\nu})\tau_{k-1}(p^{\mu'+\nu})}{ p^{\nu+\mu+\mu'}}&=\sum_{\ell\geq 0 }\sum_{0\leq \mu+\mu'\leq \ell-1}\frac{ \tau_{k-1}(p^{\ell-\mu})\tau_{k-1}(p^{\ell-\mu'}) }{ p^{\ell}} 
\end{align*}
and
\begin{align*}1&=\frac{  (h_{0,k}*\mu)(1)^2}{\phi(1)}=\Big( 1-\frac{1}{p}\Big)^{2k-2}\sum_{\mu,\mu'\geq 0}
\frac{\tau_{k-1}(p^{\mu})\tau_{k-1}(p^{\mu'}) }{ p^{\mu+\mu' }} 
\cr&=\Big( 1-\frac{1}{p}\Big)^{2k-1}\sum_{\mu,\mu'\geq 0}
\frac{\tau_{k-1}(p^{\mu})\tau_{k-1 }(p^{\mu'}) }{ p^{\mu+\mu' }}\sum_{\nu\geq 0}\frac{1}{ p^{\nu}}
\cr&=\Big( 1-\frac{1}{p}\Big)^{2k-1}\sum_{\ell\geq0}\sum_{0\leq \mu+\mu' \leq \ell}
\frac{\tau_{k-1}(p^{\mu})\tau_{k-1 }(p^{\mu'}) }{ p^{\ell}}.
 \end{align*}
Substituting these two formulas in \eqref{som1tautau}, we obtain that
\begin{align*} 
\sum_{\nu\geq 0}    \frac{  (h_{0,k}*\mu)(p^\nu)^2}{\phi(p^\nu)}
  =\Big( 1-\frac{1}{p}\Big)^{2k-1}&\sum_{\ell\geq0}\Big\{\sum_{0\leq \mu+\mu' \leq \ell}
\frac{\tau_{k-1}(p^{\mu})\tau_{k-1 }(p^{\mu'}) }{ p^{\ell}}\cr&\quad+ \sum_{0\leq \mu+\mu'\leq \ell-1}\frac{ \tau_{k-1}(p^{\ell-\mu})\tau_{k-1}(p^{\ell-\mu'}) }{ p^{\ell}} \Big\}.\end{align*}
In the second sum in brackets, the exponents of the powers of $p$ in the arguments of $\tau_{k-1}$ satisfy
$$ ( {\ell-\mu})+({\ell-\mu'})\geq \ell+1.$$ 
Hence, changing $(\ell-\mu,\ell-\mu') $ to $( \mu, \mu') $, we obtain that
\begin{align*}\sum_{\nu\geq 0}   \frac{  (h_{0,k}*\mu)(p^\nu)^2}{\phi(p^\nu)}&=
\Big( 1-\frac{1}{p}\Big)^{2k-1}\sum_{\ell\geq0}
\frac{1}{ p^{\ell}}\Big(\sum_{0\leq \mu  \leq \ell}\tau_{k-1}(p^{\mu})\Big)^2  \cr&=\Big(1-\frac{1}{p}\Big)^{2k-1}\sum_{\ell\geq 0 }\frac{  \tau_{k}(p^{\ell})^2}{ p^{\ell}}, 
\end{align*}
which in turn implies that $C_{0,0}(k)=C_k$. 
\end{proof}

We are now ready to prove Theorem \ref{thmomenttauk2}.

\begin{proof}[Proof of Theorem \ref{thmomenttauk2}]
Our starting point is the estimate \eqref{estM2DELTA}. By definition of $u_{j,k}$ (see \eqref{eqdefujk}), we obtain that the second term of \eqref{estM2DELTA} satisfies
\begin{align*} S_k(x,R)&=x\sum_{0\leq j,j' < k}  \frac{S_{j,j'}(R)}{(k-1-j)!(k-1-j')!}\int_0^x(\log t)^{2k-j-j'-2}\d t  
\\&=x\sum_{0\leq j,j' < k} {2k-2-j-j' \choose k-1-j} \sum_{ \ell=0 }^{2k-2-j-j' }(-1)^{\ell+j+j'}  \frac{(\log x)^{\ell}}{\ell!} 
S_{j,j'}(R),
  \end{align*} 
with
$$S_{j,j'}(R):=\sum_{\substack{ r\leq R  }}   \frac{ (h_{j,k}*\mu)(r)(h_{j',k}*\mu)(r)  }{\phi(r)}.$$

Let us now estimate the sums $ S_{j,j'}(R)$.  
Thanks to \eqref{precisionhj}, we can approximate these sums by substituting $h_{j,k}$ and $h_{j',k}$ with $g_{j,k}$ and $g_{j',k}$ with $g_{j,k}=h_{0,k}(-\log )^j/j!$. By this estimate and Proposition \ref{proptauk}, the induced error term is 
\begin{align*}\ll& (\log R)^{j+j'-1}\sum_{\substack{ r\leq R  }}   \frac{ (\tau_k * \widetilde h_k *\mu)(r)^2  }{\phi(r)}\\&=  (\log R)^{j+j'-1}\sum_{\substack{ r\leq R  }}   \frac{ (\tau_{k-1} * \widetilde h_k )(r)^2  }{\phi(r)} \ll (\log R)^{(k-1)^2+j+j'-1}. \end{align*}

We use the formulas
$$(g_{j,k}*\mu)(r)=\sum_{\ell\mid r} g_{j,k}(\ell)\mu(r/\ell),\qquad 
(g_{j',k}*\mu)(r)=\sum_{\ell'\mid r} g_{j,k}( \ell')\mu(r/\ell'),$$
and parametrize $\ell=dm$, $\ell'=d'm$ with $(d,d')=1$, and end up with the expression
\begin{align*}F_{j,j'}(s)&=\sum_{\substack{ r=1  }}^\infty   \frac{ (g_{j,k}*\mu)(r)(g_{j',k}*\mu)(r)  }{\phi(r)r^{s-1}}
\cr&=\sum_{\substack{ (d,d',m) \\(d,d')=1 }}  \mu( d)\mu( d') \frac{  g_{j,k}(dm)  g_{j',k} (d'm)  }{\phi(  dd'm)(dd'm)^{s-1} }\sum_{\substack{ n \\(n,dd')=1}}\frac{\mu(n )^2\phi(   m)}{\phi(n  m)n^{s-1}}.\end{align*}
We write the inner sum as an Euler product:
\begin{align*} \sum_{\substack{ n \\(n,dd')=1}}\frac{\mu(n )^2\phi(   m)}{\phi(n  m)n^{s-1}}=\prod_{p\nmid dd'm}\Big(1+\frac{1}{p^s(1-1/p)}\Big)
\prod_{\substack{p\mid m\\p\nmid dd'}}\Big(1+\frac{1}{p^s }\Big),\end{align*}
hence
\begin{align*}F_{j,j'}(s)&= \sum_{\substack{ (d,d',m) \\(d,d')=1 }}  \mu( d d') \frac{  h_{0,k}(dm)  h_{0,k} (d'm)(-\log dm)^{j}(-\log dm)^{j'}  }{j!j'!\phi(  dd'm)(dd'm)^{s-1} }\cr&\qquad\prod_{p\nmid dd'm}\Big(1+\frac{1}{p^s(1-1/p)}\Big)
\prod_{\substack{p\mid m\\p\nmid dd'}}\Big(1+\frac{1}{p^s }\Big)
\cr&=\frac{1}{j!j'!}\frac{\partial^{j+j'}F}{\partial^js_1\partial^{j'}s_2} (s,0,0),\end{align*}
with
\begin{align*}F(s,s_1,s_2)&= \sum_{\substack{ (d,d',m) \\(d,d')=1 }}   \frac{\mu( d d')  h_{0,k}(dm)  h_{0,k} (d'm)  }{\phi(  dd'm)(dd'm)^{s-1}(dm)^{s_1}(d'm)^{s_2} }\cr&\qquad\prod_{p\nmid dd'm}\Big(1+\frac{1}{p^s(1-1/p)}\Big)
\prod_{\substack{p\mid m\\p\nmid dd'}}\Big(1+\frac{1}{p^s }\Big)
.\end{align*}
This series can be written as an Euler product with local factor  $F_p(s,s_1,s_2)$ at $p$, where
\begin{align*}F_p(s,s_1,s_2)&= 1+\frac{1}{p^s(1-1/p)}+
\sum_{\nu\geq 1}\frac{(1+1/p^s)h_{0,k}(p^\nu)^2}{(1-1/p)p^{\nu(s+s_1+s_2)}}+
\cr&\qquad -\Big(  \frac{ h_{0,k}(p )  }{(1-1/p)p^{ s}}+
\sum_{\nu\geq 1}\frac{ h_{0,k}(p^{\nu+1})h_{0,k}(p^\nu) }{(1-1/p)p^{\nu(s+s_1+s_2)+s}}  \Big)\big(p^{-s_1}+p^{-s_2}\big)
 .\end{align*}
Hence the function $F$ satisfies
$$F(s,s_1,s_2)=\zeta(s)\zeta(s+s_1+s_2)^{k^2}\zeta(s+s_1 )^{-k}\zeta(s+s_2 )^{-k}\widetilde F(s,s_1,s_2)$$
with a factor $\widetilde F$ which is holomorphic and bounded when $\Re e(s)\geq \tfrac56$,
$\Re e(s_1)\geq -\tfrac1{6}$ and $\Re e(s_2)\geq -\tfrac1{6}$.
Taking $(j,j')=(0,0)$ and applying Lemma \ref{Ck}, we obtain that
$$\widetilde F(1,0,0)=\lim_{s\to1}F(s,0,0)\zeta(s)^{-(k-1)^2}=C_{0,0}(k)= C_k .$$
We have
$$F_{j,j'}(s) =\frac{1}{j!j'!}\frac{\partial^{j+j'}F}{\partial^js_1\partial^{j'}s_2} (s,0,0)= {\rm Res}_{\substack{s_1=0\\ s_2=0}} \frac{F(s,s_1,s_2)}{s_1^{j+1}s_2^{j'+1}}.$$ 

%
   
By definition of $F_{j,j'}(s), $

\begin{align*} S_{j,j'}(R)&=\sum_{\substack{ r\leq R  }}   \frac{ (g_{j,k}*\mu)(r)(g_{j',k}*\mu)(r)  }{\phi(r)}+ O\big( (\log R)^{(k-1)^2+j+j'-1}\big)
\\&= \frac 1{2 \pi i}\int_{\Re(s)=2} F_{j,j'}(s) \frac{R^{s-1}}{s-1} ds + O\big( (\log R)^{(k-1)^2+j+j'-1}\big) \\
&={\rm Res}_{s=1}
 {\rm Res}_{\substack{s_1=0\\ s_2=0}} \frac{F(s,s_1,s_2)}{s_1^{j+1}s_2^{j'+1}}\frac{R^{s-1}}{s-1}+ O\big( (\log R)^{(k-1)^2+j+j'-1}\big)  .   
\end{align*}
  We therefore have that
\begin{align*}S_k(x,R)&=x(\log x)^{2k-2 } \sum_{0\leq j,j' < k}  \frac{ S_{j,j'}(R)}{(k-j-1)!(k-j'-1)!(\log x)^{j+j' }} \cr
&\quad+ O\big( x(\log x)^{2k-2}(\log R)^{k^2-2k}\big) \cr
 = x(\log x)^{2k-2 }&\sum_{0\leq j,j' < k}{\rm Res}_{s=1}{\rm Res}_{\substack{s_1=0\\ s_2=0}}\frac{ s_1^{-j-1}s_2^{-j'-1}R^{s-1}F(s,s_1,s_2)}{(k-j-1)!(k-j'-1)!(\log x)^{j+j' } (s-1)}\cr
&\quad+ O\big( x(\log x)^{2k-2}(\log R)^{k^2-2k}\big) \cr
\\ = C_k  x(\log x)^{2k-2 }&  \sum_{0\leq j,j' < k}{\rm Res}_{s=0}{\rm Res}_{\substack{s_1=0\\ s_2=0}}\frac{ R^{s} s_1^{-j-1}s_2^{-j'-1}(s+s_1)^k(s+s_2)^k (s+s_1+s_2)^{-k^2} }{(k-j-1)!(k-j'-1)!(\log x)^{j+j' } s^2}\cr
&\quad+ O\big( x(\log x)^{2k-2}(\log R)^{k^2-2k}\big) \cr
\\ =  C_k x(\log x)^{2k-2 } &(\log R)^{(k-1)^2} \sum_{0\leq j,j' < k} \frac{v^{j+j'}c_{j,j'}}{(k-j-1)!(k-j'-1)! }+ O\big( x(\log x)^{2k-2}(\log R)^{k^2-2k}\big). \cr
\end{align*}
\end{proof}

\begin{proof}[Proof of Theorem \ref{thmoment2tauk}]
We combine Theorems \ref{thmoment2tauk} and \ref{thmoment2}. Note that $u_2^*(x,R)$ satisfies the same bound~\eqref{maju*xR} as $u_1^*(x,R)$, and hence the error term in  \eqref{estM2DELTA} is $$u_2^*(x,R)^2+ \frac{U_{0,k}(x)u_2^*(x,R)}{L(x)(\log Q(x))^2}
\ll x(\log x)^{2(k-1)}\Big( \frac{R^2}{x}(\log R)^{2(k-1)}+\frac{ R(\log R)^{ k-1 }}{x^{c /k}(\log x)^2}\Big).$$
The choice of $R$ is admissible as soon as $\eta<\min(c,\tfrac 12)$.
\end{proof}

\section*{Acknowledgments}  
The second named author was supported at the IMJ-PRG by a postdoctoral fellowship of the Fondation Sciences Math\'ematiques de Paris, and at the University of Ottawa by an NSERC Discovery Grant. We would like to thank Sary Drappeau for very useful conversations, Brad Rodgers and Kannan Soundararajan for sharing their work with us, and the IMG-PRG for excellent working conditions. 
%

\end{document}